\documentclass{amsart}
\usepackage{amsmath,amscd}
\usepackage{epsf}
\usepackage{amssymb, latexsym}
\usepackage{hyperref}
\newtheorem{theorem}{Theorem}[section]
\newtheorem{proposition}[theorem]{Proposition}

\newtheorem{corollary}[theorem]{Corollary}
\newtheorem{example}[theorem]{Example}

\newtheorem{definition}[theorem]{Definition}
\theoremstyle{remark}
\newtheorem{remark}[theorem]{Remark}

\newcommand{\mC}{\mathcal{C}}
\newcommand{\mB}{\mathcal{B}}
\newcommand{\mD}{\mathcal{D}}
\newcommand{\mF}{\mathcal{F}}
\newcommand{\mH}{\mathcal{H}}

\newcommand{\mG}{\mathcal{G}}

\newcommand{\mI}{\mathbb I}

\newcommand{\mN}{\mathbb{N}}

\newcommand{\Hom}{\text{Hom}}
\newcommand{\Id}{\text{Id}}
\newcommand{\Imf}{\text{Im}}
\newcommand{\og}{\otimes_\mG^{}}
\newcommand{\op}{\otimes^{'}}
\newcommand{\g}{_\mG^{}}

\begin{document}
\title{Induced Monoidal structure from the functor}
\author{Neha Gupta}
\email{neha.gupta@snu.edu.in,pradip.kumar@snu.edu.in,pmishra.math@gmail.com}
\author{Pradip kumar}
\address{Department of Mathematics\\
Shiv Nadar University, Dadri\\ U.P. 201314, India\\}
\begin{abstract}
Let $\mB$ be a subcategory of a given category $\mD$. Let $\mB$ has monoidal structure.  In this article, we discuss when can one extend the monoidal structure of $\mB$ to $\mD$ such that $\mB$ becomes a sub monoidal category of monoidal category $\mD$. Examples are discussed, and in particular, in an example of loop space, we elaborated all results discussed in this article.
\end{abstract}
\maketitle
\section{Introduction}
We start from a few examples of a category $\mB$
\begin{enumerate}
\item  Let $\mB$ be the category of one dimensional closed oriented manifolds without boundary and arrows as the oriented cobordism classes. It has monoidal structure given by disjoint union and empty set, see \cite{kock}. Let $\mD$ be  the category whose objects are the oriented cobordism classes between objects of $\mB$ with arrows defined in a special way (details are discussed in the example \ref{ex:loopspace_image}).
\item Let $\mB$ be the category of finite dimensional vector spaces over $\mathbb R$ with usual tensor product of vector spaces. If $\mD$ be the category having objects as loop spaces and morphisms defined in a special way (details are discussed in the section 5).
\end{enumerate}
For both these setups, we ask a natural question, that from $\mB$, can we get a monoidal structure on the category $\mD$. In this article we discussed about this type of question.  More precisely,  in theorem \ref{thm: mainthm}, we proved that under certain conditions, we can provide a monoidal structure on a category $\mD$ from  a given monoidal structure on some category $\mB$

When $\mB$ is a subcategory, the induced monoidal structure (as in theorem \ref{thm: mainthm}) on $\mD$ , from $\mB$ is the extension of the monoidal structure in the sense that $\mB$ becomes the monoidal subcategory of $\mD$.  This extension is not same as the  graded extension of the monoidal category discussed in \cite{cegarra}. In \cite{ponto}, Ponto Kate and Shulman Michael have discussed the concept of a normal lax symmetric monoidal functor between two monoidal categories; and using such a functor we have got an  induced monoidal structure on the category $\mD$. Our article is based on the discussion in \cite{ponto}.

In section 2, we revise the basics about the monoidal category, dual objects and strong lax symmetric monoidal functor.  Section 3 and 4 discuss the induced tensor with its  various properties and examples.  In the last section 5,  we did explicit calculation for the case of loop space.

\section{Preliminaries}
\subsection{Monoidal category}
We first set the terminology that we use throughout the article which can be found in \cite{maclane},\cite{ponto},\cite{kassel} etc. We begin with a few terms from category theory. A \textit{monoidal}  \textit{category} is a collection of data $(\mC, \otimes, \mI,a, \lambda,\rho)$
where $\mC$ is a category, $\otimes$ is a functor from $\mC \times \mC$ to $\mC$ called the tensor
product, an associator $a$ for $\otimes$ which is a family of isomorphisms
\begin{equation}\label{associater}
a^{}_{X,Y,Z}: (X\otimes Y) \otimes Z\to X\otimes (Y \otimes Z),
\end{equation}
an unit object $\mI$ of $\mC$, and a left unitor $\lambda$ and a right unitor $\rho$ with respect to $\mI$ which are also a family of isomorphisms
\begin{equation}\label{left unitor}\lambda^{}_X:  \mI \otimes X\to X\end{equation}
\begin{equation}\label{right unitor}\rho^{}_X:  X \otimes \mI\to X\end{equation}
such that these isomorphisms $(a,\lambda,\rho)$ satisfy certain coherence conditions (for details we refer \cite{kassel}).  To name, these conditions are the Pentagon identity and the Triangle  identities. Note that the existence of the associator ensures that the tensor product $\otimes$ is associative up to isomorphism and the existence of the unitors ensure that $\mI$ serves as a unit for $\otimes$, up to isomorphism.

Let $X$ and $Y$ be objects in $\mC$. We say $X$ has a \textit{left dual} $Y$ (or $Y$ has a \textit{right dual} $X$) if there
is
\begin{enumerate}
\item[$\cdot$] a morphism $ev : Y \otimes X \to \mI$ (called evaluation), and
\item[$\cdot$]  a morphism $coev : \mI \to X\otimes Y$ (called coevaluation)
\end{enumerate}
such that the compositions
\[
\begin{CD}
X@>{\lambda^{-1}_X}>>
 \mI\otimes X @>coev\otimes \Id_X>> (X\otimes Y)\otimes X@>a^{}_{X,Y,X}>>  X\otimes (Y\otimes X)@>\Id_X\otimes ev>>  X\otimes \mI@>\rho^{}_X>>  X
\end{CD}
\]
\text{ and\\}
\[
\begin{CD}
Y@>{\rho^{-1}_Y}>>
 \mI\otimes Y @>\Id_Y\otimes coev>> Y\otimes (X\otimes Y)@>a^{-1}_{Y,X,Y}>>  (Y\otimes X)\otimes Y@>ev\otimes \Id_Y>>  \mI\otimes Y@>\lambda^{}_Y>>  Y
\end{CD}
\]

are both identity. If $X$ has a left [right] dual, we say that $X$ is left [right] dualizable, If the left and right duals are isomorphic in $\mC$, we say $X$ is \textit{dualizable}, and in general the dual of $X$ is denoted as $X^*$.

Let $\mC$ and $\mD$ be symmetric monoidal categories.  There is a special functor which takes dualizable object in $\mC$ to dualizable object in $\mD$.  Below we give the precise definition of such a functor and recall required  properties (for detail we refer to  \cite{ponto}).
\begin{definition}[Lax Symmetric monoidal functor]
Lax symmetric monoidal functor between symmetric monoidal  categories consists of  a functor $\mF: \mC\to \mD$ and natural transformations
$$c: \mF(M)\otimes \mF(N)\to \mF(M\otimes N)$$
$$i: \mI_\mD\to \mF(\mI_\mC)$$
satisfying appropriate coherence axioms.
\end{definition}
$\mF$ is normal if $i$ is an isomorphism and $\mF$ is strong if both $c$ and $i$ are isomorphism.
\begin{proposition}[Prop.6.1, \cite{ponto}]\label{ponto_1}
If $\mF: \mC\to \mD$ be normal lax symmetric monoidal functor, let $M\in Ob(\mC)$ be dualizable with dual $M^*$ and assume that $c: \mF(M)\otimes \mF(M^*)\to \mF(M\otimes M^*)$ is an isomorphism then $\mF(M)$ is dualizable with dual $\mF(M^*)$.
\end{proposition}
\section{Monoidal structure on Image of $\mF$}
Let $\mC$ and $\mD$ be two categories where $\mC$ has a monoidal structure given by $(\otimes, a, \lambda, \rho)$. Let $\mF: \mC\to \mD$ be a functor, then $\Imf(\mF)$ which is the collection of the images of all objects and morphisms in $\mC$ under $\mF$ becomes a subcategory of $\mD$.  With this set up we wish to give a monoidal structure on $\Imf(\mF)$. If $\mF$ satisfies certain conditions, which we discuss below, then the category $\Imf(\mF)$ has monoidal structure induced from $\mC$.

Let the functor $\mF$ satifies following conditions: for any pair of objects $a_1$, $a_2$ and $a_3$, $a_4$ in $\mC$, such that  whenever we have
\begin{equation}
\mF(a_1)= \mF(a_2)\text{ and }
\mF(a_3)= \mF(a_4)
\end{equation}
then $$\mF(a_1\otimes a_3)= \mF(a_2\otimes a_4)$$

We have similar conditions on morphisms also. For any pair $f, f'$ and $g,g'$  of morphisms in $\mC$, such that
\begin{equation}
\mF(f)= \mF(f')\text{ and }
\mF(g)= \mF(g')
\end{equation}
\text{ then, we have} $$\mF(f\otimes g)= \mF(f'\otimes g')$$

We use the notation $\mB$ for $\Imf{(\mF)}$ and we  define the monoidal functor $\otimes_{\mF}$ on $\mB$ where, $$\otimes_{\mF}:\mathcal{\mB}\times \mathcal{\mB}\to \mathcal{\mB}$$ as follows: on objects,
$$\otimes_{\mF}^{}(b_1, b_2):= \mF(a_1\otimes a_2)$$
where $b_1=\mF(a_1)$ and $b_2=\mF(a_2)$. The condition (4) on functor ensures that this is well defined on objects. On morphisms,
$$\otimes_{\mF}(f_1, f_2):= \mF(h_1\otimes h_2)$$
where $f_1=\mF(h_1)$ and $f_2=\mF(h_2)$. Condition (5) will make sure that this is also well defined as a map. Compatibility of $\otimes_\mF^{}$ with composition in $\mB$ will be followed from the compatibility of the tensor $\otimes$ with composition in $\mC$. Again, the compatibility of $\otimes_\mF^{}$ with identity maps of $\mB$ will require only the compatibility of the tensor $\otimes$ with identity maps of $\mC$. Explicitly, we want to show that the equation $\Id_{b_1^{}\otimes b_2}^{}=\Id_{b_1^{}}^{}\otimes_\mF^{} \Id_{b_2^{}}^{}$ is true in $\mB$ where $b_1^{}$ and $b_2^{}$ are objects in $\mB$, such that $b_\mI^{}=\mF(a_1)$ and $b_2^{}=\mF(a_2)$, for some objects $a_1$ and $a_2$ in $\mC$. Then $b_1\otimes_\mF b_2=\mF(a_1^{}\otimes a_2^{} )$ and the easy calculations shown below gives the result.

\begin{eqnarray*}
\Id_{b_1^{}}^{}\otimes_\mF^{} \Id_{b_2^{}}^{} &=& \Id_{\mF(a_1^{}})^{}\otimes_\mF^{} \Id_{\mF(a_2^{})}^{}\\
&=& \mF(\Id_{a_1^{}}^{})\otimes_\mF^{} \mF(\Id_{a_2^{}}^{})\\
&=& \mF\big(\Id_{a_1^{}}^{}\otimes \Id_{a_2^{}}^{}\big)\\
&=& \mF\big(\Id_{a_1^{}\otimes a_2^{} }^{}\big)\\
&=& \Id_{\mF(a_1^{}\otimes a_2^{} )}^{}\\
&=& \Id_{b_1^{}\otimes b_2^{}}^{}.\\
\end{eqnarray*}

This defines a tensor product $\otimes_{\mF}$ on $\mB$. For $\mB$ to have a  monoidal structure we want the left and right unitors as well as the associator in $\mB$. These are nothing but $\mF(a)$, $\mF(\lambda)$, $\mF(\rho)$ where $a$, $\lambda$ and $\rho$ are the respective left unitor, right unitor and associator of $\mC$ and they will satisfy the triangle and pentagon axioms since they are true in $\mC$. Thus we have proved the following.

\begin{theorem}\label{Thm:Tensor_on_ImF}
If $\mF$ is a functor from a monoidal category $(\mC, \otimes)$ to a category  $\mD$ that satisfies the following conditions
\begin{enumerate}
\item  for any objects  $a_1, a_2, a_3, a_4$ in $\mC$ such that
$\mF(a_1)= \mF(a_2)\;\;and\;\; \mF(a_3)= \mF(a_4)$, then we have $\mF(a_1\otimes a_3)= \mF(a_2\otimes a_4)$,
\item
for any morphisms $f, g, f',g'$ in $\mC$ if
$\mF(f)= \mF(f')= h;\;\; \mF(g)= \mF(g')= k\text{ then, we have} $
$\mF(f\otimes g)= \mF(f'\otimes g')$.
\end{enumerate}
Then $\mF$ induces a monoidal structure on $\Imf(\mF)$.
\end{theorem}

\begin{example}
Let $\mathcal{D}$ be a monoidal  category with tensor $\otimes_D$ and $\mathcal{C}= \mathcal{D}\times \mathcal{D}$, then $\mC$ is also a monoidal category with obvious tensor, defined component-wise.  Define a functor
$\mF: \mC\to \mD$  such that it associates to each object $(X,Y)$ in $\mC$ to $X$ in $\mD$ and to each morphism $(f,g)\in Hom(X_1,Y_1), (X_2, Y_2))$ to $f\in Hom(X_1, Y_1).$

 One can think of the functor $\mF$ as a projection onto the first component. Further, $\mF$ also satisfies both the condition of theorem \ref{Thm:Tensor_on_ImF}.  Therefore we have a monoidal structure $\otimes_\mF$ on $\mD$. Moreover, both monoidal structures $\otimes_D$ and $\otimes_\mF$ are same.
\end{example}
\begin{example}\label{ex:loopspace_image} (See, \cite{kock})
Let $\mathcal{A}$ be the category having 1-dimensional compact oriented manifold without boundary as objects along with the empty set. Morphism between any two objects $X$ and $Y$ is the oriented cobordism class. We write a representative of this class and also the class as $(W,X,Y)$ where $W$ denotes a 2-dimensional oriented cobordism between $X$ and $Y$. Composition of morphism is given by gluing the representatives of the two cobordism classes. Kock in \cite{kock} has discussed that the result of this gluing does not depend on the actual cobordisms chosen; rather only their class. The class of cylinders is the identity for the composition. Monoidal structure on $\mB$ is the disjoint union of the objects and the of the cobordisms. The Tensor between two morphisms $(W_1,X_1,Y_1)$ and $(W_2,X_2,Y_2)$ is a cobordism class $(W_1\sqcup W_2,X_1\sqcup X_2,Y_1\sqcup Y_2 )$. The unit object of $\mB$ is the empty set.

Consider the product category $\mathcal{A}\times \mathcal{A}$ and call it as $\mB$. Use the tensor on $\mathcal{A}$ to give a monoidal structure on $\mB$. It is defined in the obvious way as component-wise. In particular, the unit object here is $(\emptyset,\emptyset)$.

We now define a category $\mD$ with objects as oriented 2-dimensional cobordism classes, which have oriented cobordisms between the objects of $\mathcal{A}$ but without holes only. As before, we write an object of $\mD$ as $(W,X,Y)$. A morphism between two objects $(W_1,X_1,Y_1)$ and $(W_2,X_2,Y_2)$ in $\mD$ is a pair of morphisms in $\mathcal{A}$, one between $X_1$ and $X_2$ and the other between $Y_1$ and $Y_2$. Since a  morphism of $\mD$ is a pair $(W_1, W_2)$ so composition of morphism is the pair wise gluing of the cobordism classes.

We define a functor $\mF:\mB\to\mD$ such that any object $(X,Y)$ in $\mB$ is taken to the unique cobordism class of oriented 2-dimensional cobordisms between $X$ to $Y$ (without holes). $\mF$ is identity on morphisms.

Clearly this functor is fully faithful. That is to say that $\Imf(\mF)$ is whole of $\mD$. Further, this functor satisfies the two conditions of the above theorem \ref{Thm:Tensor_on_ImF}. Hence we get an induced tensor on $\mD$. We discuss this induced functor explicitly later in examples \ref{tensor on compact}.
\end{example}

\begin{corollary}
Let $\mF: \mC\to \Imf(\mF)$ be a functor as in theorem \ref{Thm:Tensor_on_ImF}.  If $(\mC, \otimes)$ is symmetric then the category $(\Imf(\mF), \otimes_\mF)$ is symmetric as well.
\end{corollary}
\begin{proof}
We need to give a symmetric structure in $\mD$  and this will be simply be the image of the symmetry of $\mC$  under the functor $\mF$. Explicitly, if $D_1, D_2\in Ob(Im(\mF))$, then there are $C_1, C_2\in Ob(C)$ such that $\mF(C_1)= D_1$ and $\mF(C_2)= D_2$.
Since $\mC$ is symmetric, therefore there is a symmetric map, say $s_{C_1,C_2}: C_1\otimes C_2\to C_2\otimes C_1$ which is an isomorphism.  Consider $\mF(s_{C_1,C_2}^{}): \mF(C_1\otimes C_2)\to \mF(C_2\otimes C_1)$.
As $\mF$ is a functor therefore $\mF(s_{C_1,C_2})$ is an isomorphism between   $D_1\otimes_\mF D_2:= \mF(C_1\otimes C_2)$ and $D_2\otimes_\mF D_1:= \mF(D_1\otimes D_1)$.
\end{proof}
\begin{corollary}
A functor $\mF$ between a symmetric monoidal category $\mC$ to $\Imf(\mF)$ where $\Imf(\mF)$ has the induced monoidal structure from $\mC$ and which satisfies the conditions of theorem (3.1) is a strong lax symmetric monoidal functor.
\end{corollary}
\begin{proof}
This is direct from the definition of the induced monoidal structure on $\Imf(\mF)$ described in the theorem \ref{Thm:Tensor_on_ImF}, which is, $D_1\otimes_\mF D_2= \mF(C_1\otimes C_2)$.  Thus by definition itself we get $$\mF(C_1)\otimes_\mF \mF(C_2)= \mF(C_1\otimes C_2)$$
And $$\mI_{\Imf(\mF)}= \mF(\mI_\mC).$$
\end{proof}

\begin{corollary}
If $A\in Ob(\mC)$ is dualizable in $(\mC,\otimes)$, then $\mF(A)$ is dualizable in $(\Imf(\mF),\otimes_\mF)$.
\end{corollary}
\begin{proof}
Since $\mF$ is a strong lax symmetric monoidal functor, then by proposition \ref{ponto_1}, we have the result directly.
\end{proof}

\section{Producut category and  Base}

\subsection{Product Category over Index set $\Lambda$}

Let $\Lambda$ be an indexing set. Abusing the use of notation, we define an indexing category, denoted by $\Lambda$, with objects as elements of the indexing set $\Lambda$, and morphisms as
\[ \Hom(\lambda_1, \lambda_2)=\begin{cases}
      \{\Id_{\lambda}^{}\} & \text{if } \lambda_1= \lambda_2 \\
      \emptyset & \text{otherwise}
   \end{cases}
\]

Composition is trivial, which will again be the identity map of the source object.
In particular, we will be considering finite sets as well as countably infinite sets (like set of natural numbers) as a category wherever we talk of a product category in the rest of the paper.

Let $\mB$ be a monoidal category with a unit object $\mI$. We define the  \textit{product} $\Pi_{\Lambda}(\mB)$ of a category $\mB$ over an indexing category $\Lambda$ as a collection of functors $\mF$ from $\Lambda$ to $\mB$ with $\mF(\lambda)$ being an object in $\mB$, and $\mF(\Hom(\lambda, \lambda))=\{\Id_{\mF(\lambda)}^{}\}$. The collection is non empty since  $C$ (the constant functor) that maps every object in $\Lambda$ to the identity $\mI$ in $\mB$ and $\mC(\Id_\lambda)$ is the identity map of the unit $\mI$ in $\mB$ for $\lambda\in\Lambda$. The \textit{product category} $\mD'$ of $\mB$ over an indexing set $\Lambda$  is defined as follows:
\begin{enumerate}
\item an object in $\mD'$ is a functor from $\Lambda \to \mB$,
\item a morphism between two objects $\mG,\mF\in Ob(\mD')$ is a natural transformation $\eta: \mF \to \mathcal{G}$ which gives a family of morphisms $\eta_\lambda^{}: \mF(\lambda)\to \mathcal{G}(\lambda)$ in $\mathcal{B}$,
\item composition of morphism is the vertical composition of two natural transformations. That is, if $\eta\in Hom(\mF, \mathcal{G})$ and $\zeta\in Hom(\mathcal{G}, \mathcal{H})$ are two vertically composable natural transformations, then the composition $\zeta\circ \eta$ is a natural transformation between $\mathcal{F}$ and $\mathcal{H}$ given by $\zeta\circ\eta(\lambda)= \zeta(\lambda)\circ\eta(\lambda)$ for every $\lambda\in\Lambda.$
\end{enumerate}

In particular, if the indexing set $\Lambda$ has only two elements, then as expected, the product category $\mD^{'}$ of $\mB$ over $\Lambda$ is isomorphic to $\mB\times\mB$.

\begin{example}
Finite product.  In case the indexing set is a finite set, then the product category $\mD^{'}$ of any category $\mB$ would be finite copies of $\mB$. 
\end{example}

\subsection{Base}
Let $\mD$ be a category. We define a base for $\mD$ as the data $(\mB, *, \mG)$, where $\mB$ is a category, the functor $*$ (we call as the pasting functor) from $\mD^{'}\to\mD$, and $\mG:\mD\to\mD^{'}$, the decomposition functor (here $\mD^{'}$ is the product category of $\mB$ over some indexing set, say $\Lambda$) such that $*\mG\cong\text{Id}_{\mD}$. $\mB$ in base, we call as base category for $\mD$.  This is like asking for a representation of every object of $\mD$ in terms of objects in $\mB$.  Thus for the same category $\mB$ with different indexing sets in the product category, might give different bases for $\mD$.

\begin{example}
Every category is a trivial base of itself. Simply take $\mB= \mD$, $\Lambda=\{1\}$ with $\mG$ and $*$ as identity functors.
\end{example}
\begin{example}\label{tensor on compact}(continuation to example \ref{ex:loopspace_image}). Countable product of the category $\mathcal{A}$ of one dimensional compact oriented manifold without boundary over $\Lambda=\{1,2,\cdots\}$ serves as a base for the category $\mD$ of oriented 2-dimensional cobordism classes, which have oriented cobordisms between the objects of $\mathcal{A}$ but without holes only. As before, we write an object of $\mD$ as $(W,X,Y)$. The pasting functor $*:\mD^{'}\to \mD$ is defined such that $*(\mH)$ is the cobordism between $\mH(1)$ and $\mH(2)$ for any functor $\mH$ in $\mD^{'}$. The decomposition functor $\mG$ picks an object in $\mD$, which is an oriented cobordism class , say $(W,X,Y)$ without holes, and maps it to the functor : $1\mapsto X$, $2\mapsto Y$ and the rest of the elements of the indexing category are mapped to the empty set. The combination of the two functors is such that $*\mG=\Id_{\mD}$.
Note that the pasting functor $*:\mD^{'}\to\mD$ acts in a way that its restriction onto two copies of $\mathcal{A}$ (that is, indexing category with only two objects) becomes the same functor $\mF$ as discussed in example \ref{ex:loopspace_image}.
\end{example}

\begin{proposition}If the base $\mB$ is a monoidal category then so is the product category $\mD'=\Pi_\Lambda\mB$.
\end{proposition}
\begin{proof}
For defining a bifunctor (tensor product) $\otimes' : \mD'\times \mD'\to \mD'$ we have the following data: for any two objects (which are functors) $\mF$ and $\mG$ in $\mD^{'}$,
\begin{enumerate}
\item On level of objects: $\mF\otimes' \mG$ is a functor from $\Lambda\to \mathcal{B}$ defined as
\begin{enumerate}
\item On object level $$\mF\otimes'\mathcal{G}(\lambda)= \mathcal{F}(\lambda)\otimes_B\mathcal{G}(\lambda)\quad\text{ for }\lambda\in\Lambda.$$
\item On the level of morphisms:
$$\mathcal{F}\otimes'\mathcal{G}(\Id_\lambda)= \mathcal{F}(\Id_\lambda)\otimes_\mB\mathcal{G}(\Id_\lambda)= \Id_{\mF(\lambda)}\otimes_\mB \Id_{\mathcal{G}(\lambda)}=  \Id_{\mathcal{F}(\lambda)\otimes_B\mathcal{G}(\lambda)}$$
\end{enumerate}
\item On the level of morphism: For $i=1,2$, if $\eta_i^{}$ are two natural transformations from $\mF_i\to \mathcal{G}_i$, then $\eta_1\otimes'\eta_2: \mF_1\otimes'\mF_2\to \mathcal{\mG_1}\otimes'\mathcal{\mG_2}$ is a natural transformation which is given by the collection of morphisms $$\big(\eta_1\otimes'\eta_2\big)_\lambda^{}: \mF_1(\lambda)\otimes_B\mF_2(\lambda)\to \mG_1(\lambda)\otimes_B\mG_2(\lambda)$$
where each of these morphisms are defined as $(\eta_1)_\lambda\otimes_\mB(\eta_2)_\lambda$ for each $\lambda\in\Lambda$.
\item Composition of morphisms: For two pair of composable transformations
$
\begin{CD}
\mF_1^{} @>\eta_1^{}>> \mG_1^{} @>\zeta_1^{}>>\mH_1
\end{CD}
\text{ and }
\begin{CD}
\mF_2^{} @>\eta_2^{}>> \mG_2^{} @>\zeta_2^{}>>\mH_2
\end{CD}
$, we want to show that their composition
$$(\eta_1\otimes'\eta_2)\circ(\zeta_1\otimes'\zeta_2)= (\eta_1\circ\zeta_1)\otimes' (\eta_2\circ \zeta_2)$$ holds in $\mD^{'}$. This equation of natural transformation is basically a collection of equations of morphisms in $\mB$. Each equation is true using the composition of morphisms in $\mB$.
\item Identity morphism:
$$\Id_{\mF\otimes'\mathcal{G}}^{}= \Id_\mathcal{F}\otimes'Id_\mathcal{G}$$
This is true using the monoidal structure of $\mB$.
\end{enumerate}

The associator $a^{'}$, as well as the left and right unitors,  $\lambda^{'}$ and $\rho^{'}$, are all transformations with the unit functor $I^{'}$ in $\mD^{'}$, which on objects is the collection $\{\lambda\mapsto I_{\mB};\quad\forall\lambda\in\Lambda\}$. They will satisfy the triangle and pentagon axioms since they are true in $\mB$.
\end{proof}

\begin{remark}  We observe that  $\mD'$ is symmetric if and only if $\mB$ is symmetric.
\end{remark}
\begin{theorem}\label{thm: mainthm}
There exists a monoidal structure on a category $\mD$ if it has a  base $(\mB, *, \mG)$ such that
\begin{enumerate}
    \item the base category $\mB$ has monoidal structure and
    \item the pasting functor $*$ satisfies both the conditions of theorem \ref{Thm:Tensor_on_ImF}. Explicitly, for any pair of objects $\mH_1$, $\mH_2$ and $\mH_3$, $\mH_4$ in $\mD^{'}$, such that  whenever we have
\begin{equation*}
*(\mH_1)= *(\mH_2)\text{ and }
*(\mH_3)= *(\mH_4)
\end{equation*}
then $$*(\mH_1\otimes \mH_3)= \mF(\mH_2\otimes \mH_4)$$

And, a similar condition on morphisms. That is, for any pair $\eta, \eta^{'}$ and $\zeta,\zeta^{'}$  of morphisms in $\mD^{'}$, such that

\begin{equation*}
*(\eta)= *(\eta')\text{ and }
*(\zeta)= *(\zeta')
\end{equation*}
\text{ then, we must have} $$*(\eta\otimes\zeta)= *(\eta^{'}\otimes \zeta^{'}).$$
\end{enumerate}
Moreover if $\mB$ is symmetric then with induced monoidal structure  $\mD$ will be symmetric.
\end{theorem}

\begin{proof}
A monoidal structure on $\mD$ consists of a tensor functor along with an identity object, associativity and a pair of left and right unitors in $\mD$. We now give all these structures in $\mD$ one by one. The tensor functor $\otimes_\mG^{}$ on $\mD$ is defined as the composition of functors $*\op(\mG\times\mG)$. Explicitly, on any pair of objects $X$ and $Y$, and any pair of morphisms $f:X_1\to Y_1$ and $g:X_2\to Y_2$ in $\mD$, the tensor is defined as:
$$X\og Y =*\Big(\mG(X)\op\mG(Y)\Big), \text{ and}$$
$$f\og g =*\Big(\mG(f)\op\mG(g)\Big)$$
where $\op$ is the tensor on $\mD^{'}$. Since the tensor $\og$ is dependent on the tensor of $\mB$ and hence on the tensor $\op$ of $\mD^{'}$, the associativity $(a\g)$, the unit object (I$_\mG$), the left unitor $(\lambda\mG)$ and the right unitor $(\rho\g)$ are defined in the obvious way as $*(a^{'})$, *(I$^{'}$), $*(\lambda^{'})$ and $*(\rho^{'})$ respectively. These morphisms are isomorphisms as isomorphisms are preserved under the action of functor $*$, and hence they will satisfy the pentagon and triangle axioms.

If $\mB$ is symmetric then $\mD'$ is symmetric, as we have a functor $* \mD'\to \mD$,  therefore symmetric structure of $\mD'$ is carried to $\mD$ by the functor.
\end{proof}

In particular, 
suppose $\mB$ is a subcategory of a category $\mD$ such that we have the same setup as in above theorem. That is, $\mB$ has a monoidal structure $\otimes$ along with two functors $*$ and $\mG$ as defined above, then $(\mB,\otimes)$ will be a monoidal subcategory of $(\mD,\og)$.

\begin{definition}
Let $(\mB,*,\mG)$ be a base for a category $\mD$ such that $\mB$ has a monoidal structure.
Then the monoidal structure on $\mD$ coming from the base satisfying the two conditions of above theorem is denoted by $\otimes_\mG^{}$, and we call this monoidal structure on $\mD$ as induced tensor on $\mD$ from the base $\mB$. In this case, we call base as  monoidal base.
\end{definition}

\begin{example}
Let $\mB$ be the category of  finite dimensional vector spaces over $\mathbb R$ with usual tensor product $\otimes_\mathbb R$ of vector spaces (-ie- $\otimes_\mB=\otimes_\mathbb R$).  By example 4.3, taking $\mD=\mB$ we can think of $\mB$ as a base of itself with $\Lambda=\{1\}$, and the functors $\mG$ and $*$ as the identity functor. We see that $\otimes_\mG= \otimes_\mB$.
\end{example}
\begin{example} (continuation to examples \ref{ex:loopspace_image}, \ref{tensor on compact})
In example \ref{tensor on compact}, we have shown that the category $\mathcal{A}$ of one dimensional compact oriented manifold without boundary together with indexing category $\Lambda=\{1,2,\cdots\}$  serves as a base for the category $\mD$ of oriented 2-dimensional cobordism classes between objects of $\mathcal{A}$ (without holes).  Now, in particular let us choose the indexing category $\Lambda$ to be of two objects only (that is, $\mD^{'}=\mathcal{A}\times\mathcal{A}).$

In example \ref{ex:loopspace_image}, we have discussed a monoidal structure on $\mathcal{A}$ and also shown that the functor $\mF:\mathcal{A}\times \mathcal{A}\to \mD$ satisfies the two conditions asked in the theorem (4.8) above. Think of $\mF$ playing the role of the required $*$ functor for the monoidal base. Thus by the above theorem, there is an induced tensor from the monoidal base $\mathcal{A}$ onto $\mD$. Explicitly, for any two objects $(W_1,X_1,Y_1)$ and $(W_2,X_2,Y_2)$ in $\mD$ we have
$$(W_1,X_1,Y_1)\otimes_\mF^{} (W_2,X_2,Y_2)=(W,X_1\sqcup X_2,Y_1\sqcup Y_2) $$
where $W$ is a cobordism between $X_1\sqcup X_2$ and $Y_1\sqcup Y_2$ without any holes; For any two
morphisms, say, $\Big((W_1,X_1,Y_1),(W_2,X_2,Y_2)\Big)$ and $\Big((W_3,X_3,Y_3),(W_4,X_4,Y_4)\Big)$ in $\mD$ we have
$$\Big((W_1,X_1,Y_1),(W_2,X_2,Y_2)\Big)\otimes_\mF^{}\Big((W_3,X_3,Y_3),(W_4,X_4,Y_4)\Big)=$$
$$\Big((W_1\sqcup W_3,X_1\sqcup X_3,Y_1\sqcup Y_3), (W_2\sqcup W_4,X_2\sqcup X_4,Y_2\sqcup Y_4\Big).$$
Note that the unit object under the induced tensor will be the oriented 2-dimensional class of a sphere (cobordism between empty set to itself without holes!)
\end{example}

\begin{proposition}  $\mB$ is dualizable if and only if $\mD'$ is dualizable. \end{proposition}
\begin{proof}Firstly, let us assume that $\mB$ is dualizable. That means every object has a dual in $\mB$. Let $\mF: \Lambda\to \mB$ be an object in $\mD'$.  We define a functor $\mF^*: \Lambda\to \mB$ such that
\begin{enumerate}
\item For any $\lambda\in \Lambda$, we have $\mF^*(\lambda):= \mF(\lambda)^*$, here $\mF(\lambda)^*$ is the dual of $\mF(\lambda)$ in $\mB$.
\item $\mF(\Id_\lambda)= \Id_{\mF(\lambda)^*}$
\end{enumerate}
Clearly $\mF^* \in Ob(\mD')$.  We claim that $\mF^*$ is the dual of $\mF$ in the category $\mD'$ with monoidal structure as in proposition 4.4.

We have natrual transformations $ev: \mF\otimes'\mF^*\to \mI_\mD'$, defined by
$$ev_\lambda: \mF(\lambda)\otimes_\mB\mF(\lambda)^*\to \mI_\mB.$$ Here $ev_\lambda$ is the evaluation map in the monoidal category $\mB$ for $\mF(\lambda)$ and $\mF(\lambda)^*$.

Similiary, we have coevaluation map in the category $\mD'$, given by natural transformation, $coev: \mI_\mD\to \mF\otimes'\mF^*$,  it is defined by coevaluation map $coev_\lambda$ in the monoidal category $\mB$ for $\mF(\lambda)$ and $\mF^*(\lambda)$.
$$coev_\lambda:\mI_B\to \mF(\lambda)\otimes_B\mF^*(\lambda).$$

As for each $\lambda$, $\mF(\lambda)$ is dual of $\mF^*(\lambda)$,  we have

$$\begin{CD}
\mF(\lambda)@>{l^{-1}_{\mF(\lambda)}}>>
 \mI_\mB\otimes_\mB \mF(\lambda) @>coev_\lambda\otimes_\mB \Id^{}_\mF(\lambda)>> (\mF(\lambda)\otimes_\mB \mF^*(\lambda))\otimes_\mB \mF(\lambda)@>a^{}_{\mF(\lambda),\mF^*(\lambda),\mF(\lambda)}>>
 \end{CD}
 $$
 $$\begin{CD}
 \mF(\lambda)\otimes_\mB (\mF^*(\lambda)\otimes_\mB \mF(\lambda))@>\Id^{}_\mF(\lambda)\otimes_\mB ev_\lambda>>  \mF(\lambda)\otimes_\mB \mI_\mB@>r^{}_\mF(\lambda)>>  \mF(\lambda)
\end{CD}
$$
here $r$ and $l$ are right and left unitors in $\mB$ and
$$
\begin{CD}
\mF^*(\lambda)@>{r^{-1}_{\mF^*(\lambda)}}>>
 \mI_\mB\otimes_\mB \mF^*(\lambda) @>\Id^{}_{\mF^*(\lambda)}\otimes_\mB coev_\lambda>> \mF^*(\lambda)\otimes_\mB (\mF(\lambda)\otimes \mF^*(\lambda))@>a^{-1}_{\mF^*(\lambda),\mF(\lambda),\mF^*(\lambda)}>>
 \end{CD}$$

$$\begin{CD}(\mF^*(\lambda)\otimes_\mB \mF(\lambda))\otimes_\mB \mF^*(\lambda)@>ev_\lambda\otimes_\mB \Id^{}_{\mF^*(\lambda)}>>  \mI_\mB\otimes \mF^*(\lambda)@>l^{}_{\mF^*(\lambda)}>>  \mF^*(\lambda)
\end{CD}
$$

are both identity,  since $\mF(\lambda)$ has left dual $\mF^*(\lambda)$.  Similiary we can see that $\mF(\lambda)$ has same right dual.   This proves that  we have $\mF^*$ is the dual of $\mF$ in $(\mD',\otimes')$.

Conversely, for a fix $\lambda\in \Lambda$, we define a functor $EV_\lambda: \mD'\to \mB$, such that for each object $\mF\in \mD'$, it associates to $\mF(\lambda)$ in  $\mB$ and For each morphism $\eta: \mF\to \mG$, it associates to $\eta_\lambda:\mF(\lambda)\to \mG(\lambda)$.  $EV_\lambda$ is a functor such that $\mB= Im(EV_\lambda)$, and by the definition of monoidal structur on $\mD'$, we see that
$$EV_\lambda(\mF)\otimes_B EV_\lambda(\mG)=  EV_\lambda(\mF\otimes'\mG).$$
Therefore $EV_\lambda$ is  strong lax symmetric monoidal functor.  As a result every object in $\mB$ is dualizable, if $\mD'$ is dualiable.

\end{proof}

\begin{proposition}
$\mD$ is dualizable if $\mD'$ is dualizable.
\end{proposition}
\begin{proof}
Since the functor $*: \mD'\to \mD$ is a strong lax symmetric monoidal functor and $\mD= *(\mD')$, it proves the result.
\end{proof}

\section{example: Loop space}
\subsection{Loop space and lifting of maps}
Let $V$ be a finite dimensional vector space over $\mathbb R$. By loop space we mean set of all maps from $S^1$ to $V$, we denote this set by $LV$.  With usual function addition and scalar multiplication this is a vector space over $\mathbb R$.

If $T: V_1\to V_2$ is a linear map from vector space $V_1$ to $V_2$, we define a map
$$\widetilde{T}: LV_1\to LV_2\quad \text{as:}$$

\begin{equation}
\widetilde{T}(\gamma)(t):= T(\gamma(t)).
\end{equation}

We call $\widetilde{T}$ is the lifting of $T$.  $\widetilde{T}$ is a linear map from $LV_1\to LV_2$.
\subsection{Categories $\mB$ and $\mathcal{D}$}
Category $\mB$ is the category of finite dimensional vector spaces over $\mathbb R$ and linear maps as morphisms. Category $\mathcal{D}$  is defined by the following:
\begin{enumerate}
\item Objects are $LV$ for all finite dimensional vector spaces $V$ over $\mathbb R$.
\item Morphisms $Hom (LV_1, LV_2):= L Hom (V_1, V_2)$.  For $f: LV_1\to LV_2$ there is a unique linear  map  $T_f: V_1\to V_2$ such that

    \begin{equation}
    f(\gamma)(t):= \widetilde{T_f}(\gamma(t)).
    \end{equation}

\end{enumerate}
For $f\in Hom(LV_1, LV_2)$ and $g\in Hom(LV_2, LV_3)$ there exits $T_f$ and $T_g$ and we have
$$\widetilde{T_g}\circ \widetilde{T_f}= \widetilde{T_{g\circ f}}$$
This means that $g\circ f\in Hom(LV_1,LV_3)$.

\subsection{Monoidal structure on  $\mathcal{D}$ from the theorem above}
From the construction of categories, we have a natural choice of functor $L:\mathcal{B}\to \mathcal{D}$ defined as:

$$ V\to LV;\;\;f\to \widetilde{f}$$
where $\widetilde{f}$ is given by equation (6). Let $\mG: \mathcal{D}\to \mathcal{B}$ be a functor which associates each object $LV$ to $ \{\gamma\in LV:\gamma \text{ is a constant loop}\}$ and any morphism $f: LV_1\to LV_2$ to $T_f: V_1\to V_2$ given by equation (7).

Taking $\Lambda=\{1\}$ so that the product category of $\mB$ over $\Lambda$ is $\mB$ itself, then the data $(\mathcal{B}, \mG, L)$ becomes a base for $\mD$.  Here $L$ works as $*$ functor in the definition.

\begin{proposition}
There is a monoidal structure on $\mD$ induced from its base $\mB$.
\end{proposition}
\begin{proof}
First we see that for vector spaces $V_1, V_2$ such that  $LV_1=  LV_2$ and $W_1, W_2$ such that $LW_1= LW_2$, we have $$L(V_1\otimes_\mB W_1)= L(V_2\otimes_\mB W_2).$$
Also for any morphisms $\eta, \eta', \zeta, \zeta'$, if $L(\eta)= L(\eta')$ and $L(\zeta)= L(\zeta')$, we have $$L(\eta\otimes_\mB\zeta)= L(\eta'\otimes_\mB\zeta')$$

Therefore $(\mB, \mG, L)$ satisfies all the condition of theorem 4.6.  Thus, we have monoidal structure on $\mathcal{D}$ defined as
$$LV_1\otimes_\mG LV_2:= L(V_1\otimes V_2)$$
$$\eta\otimes_\mG \zeta:= \widetilde{T_\eta\otimes_\mB T_\zeta}$$
\end{proof}
With this induced monoidal structure on $LV$, it is obvious to see that the functor $L$ and $\mG$  are strong lax symmetric monoidal functors. And as a result of above remark, every object in the category $\mathcal{D}$ of loop space is dualizable.  Explicit calculation of the above observations is given in the following subsection.
\subsection{Explicit calculation for above results in the case of loop space.}
$$L\mathbb R\otimes LV= L(\mathbb R\otimes V)\simeq L(V).$$ Therefore $L\mathbb R$ is the identity element in this monoidal category $\mD$. If
$V$ is of $n$ dimensional vector space over $\mathbb R$ and if $V^*$ is the dual of $V$ in original sense,  then we have a map
$$ev: V\otimes V^*\to \mathbb R$$ defined by
$$ev\left(\sum_{i=1,j=1}^na_{ij}e_i\otimes e_j^*\right)= \sum_{i=1}^n a_{ii}$$
We also have the coevaluation map.
$$coev: \mathbb R\to V\otimes V^*$$
$$coev(1)= \sum_{i=1}^n e_i\otimes e_i^*$$

Let $\widetilde{ev}$ and $\widetilde{coev}$ be the lifting of $ev$ and $coev$ respectively, then we have
$$\widetilde{ev}: LV\otimes LV^*(= L(V\times V^*)\to L\mathbb R$$
by $$\left(t\to \gamma(t)=\sum_{i=1,j=1}^na_{ij}(t)e_i\otimes e_j^*\right)\to \left(t\to \sum_{i=1}^n a_{ii}(t)\right)$$

and
$$\widetilde{coev}: L\mathbb R\to LV\otimes LV^*:= L(V\times V^*)$$
by
$$\left(t\to \gamma(t)\right)\to \left(t\to \gamma(t)\sum e_i\otimes e_i^*\right)$$

Now we will write  left and right unitor.

$$\widetilde{\lambda_V}: L\mathbb R\otimes LV:=L(\mathbb R\otimes V)\to LV$$
$$\widetilde{\lambda_V}\left(t\to \sum a_i(t)e_i\right)= \left(t\to \sum a_i(t)e_i\right)$$

Also we have
$$\widetilde{\lambda_V}^{-1}: LV\to L\mathbb R\otimes LV$$
by$$\widetilde{\lambda_V}^{-1}\left(t\to \sum a_i(t)e_i\right)= \left(t\to \sum 1\otimes a_i(t)e_i\right)$$

The map $\widetilde{\lambda_V}$ and inverse are the lift  of the map $\lambda_V: \mathbb R\otimes V\to V$ and $\lambda_V^{-1}: V\to \mathbb R\otimes V$.

Similar way we have maps
$$\rho_V: V\otimes \mathbb R\to V$$
$$\rho_V(\sum a_i e_i\otimes v)= \sum b a_i e_i$$
If we lift this map to the loop space, we have corresponding map.
$$\widetilde{\rho}_V:= LV\otimes L\mathbb R:= L(V\otimes \mathbb R)\to LV$$ given by
$$ \widetilde{\rho_V}\left(t\to \gamma(t)=\sum a_i(t) e_i \otimes b(t)\right):= \left(t\to \sum b(t)a_i(t)e_i\right)$$
Also if we lift $\rho_V^{-1}$ we have
$$\widetilde{\rho_V}^{-1}\left(t\to \sum a_i(t) e_i\right):= \left(t\to \sum a_i(t) e_i \otimes 1\right)$$

Now we will calculate the map $$\widetilde{coev}\otimes \Id_{LV}: L\mathbb R\otimes LV \to L(V\otimes V^*)\otimes LV$$

Let $\gamma\in L\mathbb R\otimes LV:= L(\mathbb R\otimes V)$, we can write
$$\gamma(t)= \gamma_1(t)\otimes \gamma_2(t)= b(t)\otimes \sum a_i(t) e_i$$

In general, for each $\gamma\in L\mathbb R\otimes LV$, there exists $\gamma_1$ and $\gamma_2$ such that $\gamma(t)= \gamma_1(t) \otimes \gamma_2(t)$. Obviously, $\gamma_1\in L\mathbb R$ and $\gamma_2\in LV$.

We denote $\gamma_1\otimes \gamma_2:= \left(t\to \gamma_1(t)\otimes \gamma_2(t)\right)$

With this notation we have a symmetric structure $\widetilde{s}$ in $\mD$ given as
$$\widetilde{s}: LV_1\otimes LV_2\to LV_2\otimes LV_1$$
$$\widetilde{s}(\gamma_1\otimes \gamma_2)= \gamma_2\otimes \gamma_1$$

\subsubsection{Braiding and Trace } A braiding for a monoidal category consists of  a family of isomorphisms $$C^{}_{X,Y}:X\otimes Y\to Y\otimes X, \text{ for }X,Y \text{ in } \mC $$natural in $X$ and $Y$, such that the two Hexadonal identities are satisfied. A monoidal category together with a braiding is called a \textit{braided monoidal category}.

Let $X$ be a dualizable object in a braided monoidal category $\mC$, and that $f:X\to X$ is an endomorphism of $X$. The $\textit{trace} tr_X(f) \in End(\mI)$ is the composite
$$
\begin{CD}
\mI@>{coev^{}_X}>>
 X\otimes X^* @>C^{}_{X,X^*}>> X^*\otimes X@>\Id^{}_{X^*}\otimes f>>  X^*\otimes X@> ev^{}_X>>  \mI
\end{CD}
$$
We know that category of finite dimensional vector spaces over $\mathbb R$ with usual tensor product is a braided monoidal category. From the braiding (family of isomorphism) if we lift all the isomorphism to the loop space level (by tilde map discussed in section 5.1), we can easily see that $\mD$ with induced monoidal structure is braided. So we have the following.
\begin{proposition}
Category $(\mB, \otimes_\mB)$ and $(\mD, \otimes_G)$ are braided monoidal category.
\end{proposition}

Lets see the map
$$L\mathbb R\xrightarrow{\widetilde{coev}}LV\otimes LV^*\xrightarrow{Id\otimes Id} LV\otimes LV^*\xrightarrow{\widetilde{s}}LV^*\otimes LV\xrightarrow{\widetilde{ev}} L\mathbb R$$

$$\left(t\to \gamma(t)\sum e_i\otimes e_i^*\right)= \left(t\to \sum \gamma(t) e_i\otimes e_i^*\right)$$
$$\to \left(t\to \gamma(t).n\right)= n\gamma$$

This proves that $tr_{LV}(\Id)= n = \text{ dimensional of vector space}$.  For each object in the category we get a unique number that is the dimensional of vector space. This is because of the way we define the monoidal structure on $\mD$.

\end{document}